\documentclass{amsart}
\usepackage[utf8]{inputenc}

\usepackage{mathrsfs}
\usepackage{amscd,amssymb,amsopn,amsmath,amsthm,graphics,amsfonts,enumerate,verbatim,calc}

\newtheorem{theorem}{Theorem}[section]
\newtheorem{claim}[theorem]{Claim}

\newtheorem{lemma}[theorem]{Lemma}

\newtheorem{corollary}[theorem]{Corollary}

\theoremstyle{definition}
\newtheorem{definition}[theorem]{Definition}

\theoremstyle{remark}
\newtheorem{remark}[theorem]{Remark}
\newtheorem{question}[theorem]{Question}
\newtheorem{notation}[theorem]{Notation}

\newcommand{\cf}{{\rm cf}}

\newcommand{\UB}{{\rm UB}}

\newcommand{\ADS}{{\rm ADS}}

\newcount\skewfactor
\def\mathunderaccent#1#2 {\let\theaccent#1\skewfactor#2
\mathpalette\putaccentunder}
\def\putaccentunder#1#2{\oalign{$#1#2$\crcr\hidewidth
\vbox to.2ex{\hbox{$#1\skew\skewfactor\theaccent{}$}\vss}\hidewidth}}

\usepackage{hyperref}

\begin{document}

\title {Usuba's principle $\UB_\lambda$ can fail at singular cardinals}

\author[M.  Golshani]{Mohammad Golshani}

\address{Mohammad Golshani, School of Mathematics, Institute for Research in Fundamental Sciences (IPM), P.O.\ Box:
	19395--5746, Tehran, Iran.}

\email{golshani.m@gmail.com}

\author[S. Shelah] {Saharon Shelah}
\address{Einstein Institute of Mathematics\\
Edmond J. Safra Campus, Givat Ram\\
The Hebrew University of Jerusalem\\
Jerusalem, 91904, Israel\\
 and \\
 Department of Mathematics\\
 Hill Center - Busch Campus \\
 Rutgers, The State University of New Jersey \\
 110 Frelinghuysen Road \\
 Piscataway, NJ 08854-8019 USA}
\email{shelah@math.huji.ac.il}
\urladdr{http://shelah.logic.at}
\thanks{ The first author's research has been supported by a grant from IPM (No. 1400030417). The
	second author's research has been partially supported by Israel Science Foundation (ISF) grant no:
	1838/19. This is publication 1216 of second author. The authors thank the referee of the paper for his/her very careful reading of the paper and detecting some essential errors in earlier versions of the paper.}



\date{\today}

\begin{abstract}
 We answer a question of Usuba by showing that the combinatorial principle $\UB_\lambda$ can fail at a singular cardinal. Furthermore, $\lambda$ can
 be taken to be $\aleph_\omega.$
\end{abstract}

\maketitle
\numberwithin{equation}{section}
\section{introduction}
In \cite{usuba}, Usuba introduced a new combinatorial principle,
denoted $\UB_\lambda.$\footnote{See Section \ref{per} for the statement of the principle.}
He showed that $\UB_\lambda$ holds for all regular uncountable cardinals and that for singular cardinals, some very weak assumptions like
weak square or even $\ADS_\lambda$
imply it.
It is known that $\ADS_\lambda$ can fail for singular cardinals, for example if $\kappa$ is supercompact and $\lambda > \kappa$ is such that $\cf(\lambda) < \kappa$. Motivated by this results, Usuba asked  the following question:

\begin{question}
\label{usubaq}
(\cite[Question 2.11]{usuba}) Is it consistent that $\UB_\lambda$ fails for some singular cardinal $\lambda$?
\end{question}
In this paper we give a positive answer to the above question by showing that Chang's transfer principle $(\aleph_{\omega+1}, \aleph_\omega) \twoheadrightarrow (\aleph_1, \aleph_0)$ implies the failure of $\UB_{\aleph_\omega}$ if $\aleph_\omega$ is strong limit, see Theoem \ref{thm1}, where a stronger result is proved.

The paper is organized as follows. In Section \ref{per}, we present some preliminaries and results and then in Section \ref{proof}, we prove our main result.

\section{Some preliminaries}
\label{per}
In this section we present some definitions and results that are needed for the later section of this paper.
Let us start by introducing  Usuba's principle.
\begin{definition}
Let $\lambda$ be an uncountable cardinal. The principle $\UB_\lambda$ is the statement:
there exists a function $f: [\lambda^+]^{<\omega} \to \lambda^+$
such that if $x, y \subseteq \lambda^+$ are closed under $f$, $x \cap \lambda=y \cap \lambda$
and $\sup(x \cap \lambda)=\lambda$, then $x \subseteq y$ or $y  \subseteq x.$
\end{definition}
It turned out this principle has many equivalent formulations. To state a few of it,
let $S=\{x \subseteq \lambda: \sup(x)=\lambda   \}$,  $\theta > \lambda$ be large enough regular and let $\lhd$ be a well-ordering of $H(\theta)$. Then
we have the following.
\begin{lemma}
\label{lem1} (\cite{usuba})
The following are equivalent:
\begin{enumerate}
\item $\UB_\lambda$,

\item If $M, N \prec (H(\theta), \in, \lhd, \lambda, S, \cdots)$ are such that $M \cap \lambda=N \cap \lambda \in S,$ then either
$M \cap \lambda^+ \subseteq N \cap \lambda^+$ or $N \cap \lambda^+ \subseteq M \cap \lambda^+$,

\item If $M, N \prec (H(\theta), \in, \lhd, \lambda, S, \cdots)$ are such that $M \cap \lambda=N \cap \lambda \in S,$ and
$\sup(M \cap \lambda^+) \leq \sup(N \cap \lambda^+)$, then $M \cap \lambda^+$ is an initial segment of $N \cap \lambda^+.$
\end{enumerate}
\end{lemma}
The principle $\UB_\lambda$ has many nice implications. Here we only consider its relation with the Chang's transfer principles
which is also related to our work.
\begin{definition}
Suppose $\lambda > \mu$ are infinite cardinal. The Chang's transfer principle  $(\lambda^+, \lambda) \twoheadrightarrow (\mu^+, \mu)$ is the statement: if $\mathcal L$ is a countable first order language
which contains a unary predicate $U$, then for any $\mathcal L$ -structure $\mathcal M=(M, U^{\mathcal M}, \cdots)$ with $|M|=\lambda^+$
and $|U^{\mathcal M}|=\lambda$, there exists an elementary submodel $\mathcal N= (N, U^{\mathcal N}, \cdots)$ of $\mathcal M$ with
$|N|=\mu^+$ and $|U^{\mathcal N}|=\mu$.

Given an infinite cardinal $\nu,$ The transfer principle $(\lambda^+, \lambda) \twoheadrightarrow_{\leq \nu} (\mu^+, \mu)$ is defined similarly, where we allow the language $\mathcal L$ to have size at most $\nu$.
\end{definition}
The next lemma shows the relation between $\UB_{\aleph_\omega}$ and Chang's transfer principles.
\begin{lemma}
(\cite[Corollary 4.2]{usuba})
Suppose $\UB_{\aleph_{\omega}}$ holds. Then the Chang transfer principles $(\aleph_{\omega+1}, \aleph_\omega) \twoheadrightarrow (\aleph_{n+1}, \aleph_n)$
fail for all $1 \leq n < \omega.$
\end{lemma}
\begin{remark}
By \cite{shelah1008}, $(\aleph_{\omega+1}, \aleph_\omega) \twoheadrightarrow (\aleph_{n+1}, \aleph_n)$ fails for all
$n \geq 3$.
\end{remark}
Since the consistency of the  transfer principle $(\aleph_{\omega+1}, \aleph_\omega) \twoheadrightarrow (\aleph_{n+1}, \aleph_n)$ is open for $n = 1, 2,$
one can not use the above result to get the consistent failure of $\UB_{\aleph_{\omega}}$. In the next section we show that if $\aleph_\omega$ is strong limit, then $\UB_{\aleph_{\omega}}$  implies the failure of $(\aleph_{\omega+1}, \aleph_\omega) \twoheadrightarrow (\aleph_{1}, \aleph_0)$ as well, and hence by the results of \cite{levi} (see also \cite{eskew} and \cite{hayut}, where the consistency of $\text{GCH}+(\aleph_{\omega+1}, \aleph_\omega) \twoheadrightarrow (\aleph_{1}, \aleph_0)$ is proved using weaker large cardinal assumptions) $\UB_{\aleph_{\omega}}$ can fail.
We also need the following notion.
\begin{definition}
An uncountable cardinal $\kappa$ is said to be Jonsson, if for every function $f:[\kappa]^{<\omega} \to \kappa$ there exists a set $H \subseteq \kappa$ of order type $\kappa$ such that for each $n$, $f''[H]^n \neq \kappa.$
\end{definition}
\begin{notation}
Given a model $M$ and a subset $A$ of $M$, by $cl(A, M)$ we mean the least substructure of $M$ which includes $A$ as a subset.
\end{notation}
\begin{lemma}
\label{lem2} Assume $\lambda$ is a singular strong limit cardinal of cofinality $\kappa.$ Then there is a model $M_0$ with vocabulary $\mathcal L_0$ such that:
\begin{enumerate}
\item[(a)] $|\mathcal L_0|=\kappa$ and $|M_0|=\lambda^+,$

\item[(b)] if $M$ is an $\mathcal L$-structure which expands $M_0$, $|\mathcal L|=\kappa$ and $M$ has Skolem functions, then for $\alpha_1, \alpha_2 < \lambda^+,$ the following statements are equivalent:
    \begin{enumerate}
\item[$(\dag)_{\alpha_1, \alpha_2}$] for some submodels $N_1, N_2$ of $M$ we have:
\begin{enumerate}
\item[$(\alpha)$] $N_1 \cap \lambda=N_2 \cap \lambda$ is unbounded in $\lambda$,

\item[$(\beta)$] $\alpha_1 \in N_1 \setminus N_2$ and  $\alpha_2 \in N_2 \setminus N_1$.
\end{enumerate}
\item[$(\ddag)_{\alpha_1, \alpha_2}$] if $V_\ell=cl(\{\alpha_\ell\}, M) \cap \lambda$, $\ell=1, 2,$ and $V=V_1 \cup V_2$, then
\[
\alpha_1 \notin cl(\{\alpha_2\} \cup V, M) \text{~}\& \text{~} \alpha_2 \notin cl(\{\alpha_1\} \cup V, M).
\]
   \end{enumerate}
\end{enumerate}
\end{lemma}
\begin{proof}
Let $\langle  \lambda_i: i<\kappa         \rangle$ be an increasing sequence cofinal in $\lambda$
such that for all $i<\kappa, 2^{\lambda_i} < \lambda_{i+1}$. For each $0<n<\omega,$ let
\[
\langle F_{n, \alpha}: \alpha \in [\lambda_i, 2^{\lambda_i} ) \rangle
\]
enumerate all functions from $\lambda_i$ into $\lambda_i.$ Let $M_0$ be defined as follows:
\begin{itemize}
\item the universe of $M_0$ is $\lambda^+,$

\item $<^{M_0}= \{ (\alpha, \beta): \alpha < \beta < \lambda^+      \}$,

\item $c_i^{M_0}=\lambda_i$,

\item $P^{M_0}=\{\alpha: \alpha < \lambda\},$

\item $F_n^{M_0}$ is an $(n+1)$-ary function such that:
\begin{itemize}
\item if $i<\kappa, \alpha \in [\lambda_i, 2^{\lambda_i})$ and $\beta_0, \cdots, \beta_{n-1} < \lambda_i$, then
\[
F_n^{M_0}(\beta_0, \cdots, \beta_{n-1}, \alpha)=F_{n, \alpha}(\beta_0, \cdots, \beta_{n-1}),
\]

\item in all other cases, $F_n^{M_0}(\beta_0, \cdots, \beta_{n-1}, \beta_n)=\beta_n.$
\end{itemize}
\end{itemize}
We show that the model $M_0$ is as required. Clause (a) clearly holds. To show that clause (b) is satisfied, let
$M$ be an $\mathcal L$-structure which expands $M_0$, $|\mathcal L|=\kappa$ and suppose $M$ has Skolem functions. Let also $\alpha_1, \alpha_2 < \lambda^+.$

First suppose that $(\dag)_{\alpha_1, \alpha_2}$ holds, and suppose that  the models $N_1, N_2$ witness it. Let also
 $V_\ell=cl(\{\alpha_\ell\}, M) \cap \lambda$, $\ell=1, 2$. Clearly
 each $V_\ell$ is an unbounded subset of $\lambda$. Let $V=cl(V_1 \cup V_2, M) \cap \lambda$ and set
 $N^*_\ell=cl(\{\alpha_\ell\} \cup V, M)$.
 \begin{claim}
\label{cla2}
$N^*_\ell \subseteq N_\ell,$ for $\ell=1, 2.$
\end{claim}
\begin{proof}
Fix $\ell.$ Sine $\alpha_\ell \in N_\ell,$
$$V_\ell=cl(\{\alpha_\ell\}, M) \cap \lambda \subseteq N_\ell \cap \lambda.$$
On the other hand, $N_1 \cap \lambda=N_2 \cap \lambda,$ hence
$$V_{3-\ell}=cl(\{\alpha_{3-\ell}\}, M) \cap \lambda \subseteq N_{3-\ell} \cap \lambda=N_\ell \cap \lambda.$$
It follows that $V_1 \cup V_2 \subseteq N_\ell \cap \lambda,$
and hence
\[
V=cl(V_1 \cup V_2, M) \cap \lambda \subseteq N_\ell.
\]
Thus, as $\{\alpha_\ell\} \cup V \subseteq N_\ell,$ we have
\[
N^*_\ell=cl(\{\alpha_\ell\} \cup V, M) \subseteq N_\ell.
\]
The result follows.
\end{proof}
 \begin{claim}
\label{cla4}
$\alpha_1 \in N^*_1\setminus N^*_2$ and $\alpha_2 \in N^*_2\setminus N^*_1$.
\end{claim}
\begin{proof}
Fix $\ell \in \{1, 2\}.$ Clearly $\alpha_\ell \in N^*_\ell$. On the other hand, by our assumption,
$\alpha_\ell \notin N_{3-\ell}$, and by Claim \ref{cla2}, $N^*_{3-\ell} \subseteq N_{3-\ell}.$
Thus $\alpha_\ell \notin N^*_{3-\ell}$.
\end{proof}
Thus $(\ddag)_{\alpha_1, \alpha_2}$ is satisfied.

Conversely  suppose that $(\ddag)_{\alpha_1, \alpha_2}$ holds, and for $\ell=1, 2,$ set $N_\ell=cl(\{\alpha_\ell\} \cup V, M)$.
By our assumption, clause ($\beta$) of  $(\dag)_{\alpha_1, \alpha_2}$ holds.


 \begin{claim}
 \label{cla1} For $\ell \in \{1, 2\}$,
 $N_\ell \cap \lambda=V.$
 \end{claim}
\begin{proof}
Fix $\ell \in \{1, 2\}$. Clearly $N_\ell \cap \lambda \supseteq V.$ Now suppose towards a contradiction
that  $N_\ell \cap \lambda \neq V,$ and let $\gamma \in  N_\ell \cap \lambda \setminus V.$
As $M$ has Skolem functions, there are $n, \beta_0, \cdots, \beta_{n-1} \in V$ and $(n+1)$-ary function symbol $F$ in  $\mathcal L$
such that
\[
\gamma=F^M(\beta_0, \cdots, \beta_{n-1}, \alpha_\ell).
\]
As $\beta_0, \cdots, \beta_{n-1} \in V \subseteq \lambda$ and $\gamma < \lambda$, there is $i<\kappa$ such that
$\beta_0, \cdots, \beta_{n-1}, \gamma < \lambda_i$. Define an $n$-ary function $G: \lambda_i \rightarrow \lambda_i$ as follows:
\begin{center}
 $G(\xi_0, \cdots, \xi_{n-1})=$ $\left\{
\begin{array}{l}
         F^M(\xi_0, \cdots, \xi_{n-1}, \alpha_\ell) \hspace{0.8cm} \text{ if } F^M(\xi_0, \cdots, \xi_{n-1}, \alpha_\ell) < \lambda_i ;\\
         0  \hspace{3.85cm}  \text{ otherwise}.
     \end{array} \right.$
\end{center}
Note that  $G \in \{F_{n, \zeta}: \zeta  \in [\lambda_i, 2^{\lambda_i}) \}$. Let
\[
\zeta_*=\min\{\zeta: (\forall \xi_0, \cdots, \xi_{n-1} < c_i) G(\xi_0, \cdots, \xi_{n-1})=F^{M_0}_n(\xi_0, \cdots, \xi_{n-1}, \zeta)   \}.
\]
$\zeta_*$ is well-defined and is definable in $M$ (even in $M_0$) from $\alpha_\ell$, so clearly
$\zeta_* \in cl(\{\alpha_\ell\}, M)$.

As $\zeta_* \in cl(\{\alpha_\ell\}, M) \cap \lambda = V_\ell \subseteq V$ and $\beta_0, \cdots, \beta_{n-1} \in V,$ so
\[
\gamma=F^M(\beta_0, \cdots, \beta_{n-1}, \alpha_\ell) = F^M_{n, \zeta_*}(\beta_0, \cdots, \beta_{n-1}) \in V.
\]
This contradicts our initial assumption that $\gamma \in  N_\ell \cap \lambda \setminus V$.
The claim follows.
\end{proof}

\begin{claim}
\label{cla3}
$N_1 \cap \lambda=N_2 \cap \lambda$.
\end{claim}
\begin{proof}
By Claim \ref{cla1}, we have
$
N_1 \cap \lambda=V=N_2 \cap \lambda,$ which concludes the result.
\end{proof}
By Claim
\ref{cla3}, $N_1 \cap \lambda=N_2 \cap \lambda$,
 which implies clause ($\alpha$) of  $(\dag)_{\alpha_1, \alpha_2}$.
Thus  $N_1$ and $N_2$ are as required in clause $(\dag)_{\alpha_1, \alpha_2}$.

This completes the proof of the lemma.
\end{proof}

\section{$\UB_\lambda$ can fail at singular cardinals}
\label{proof}
In this section we  prove  the following theorem which answers Usuba's question \ref{usubaq}.
\begin{theorem}
\label{thm1}
Assume $\lambda$ is a singular strong limit cardinal.
$\UB_\lambda$ fails if  at least one of the following hold:
\begin{enumerate}
\item[(a)] $\lambda=\aleph_\omega$ and the Chang's transfer principle
$(\lambda^+, \lambda) \twoheadrightarrow (\aleph_1, \aleph_0)$
holds,

\item[(b)]  $\lambda > \mu \geq \cf(\lambda)$ are such that $(\lambda^+, \lambda) \twoheadrightarrow_{\leq \cf(\lambda)} (\mu^+, \mu)$
holds,

\item[(c)] $\lambda > \mu \geq \cf(\lambda)$ and for every model $M$ with universe $\lambda^+$ and vocabulary of cardinality $\cf(\lambda)$,
we can find an increasing sequence $\vec{\alpha}=\langle \alpha_i: i < \mu^+ \rangle$ of ordinals less than $\lambda^+$ such that
\[
S^M_{\vec{\alpha}}=\{i<\mu^+: cl(\{\alpha_i\}, M) \cap \lambda \subseteq cl(\{\alpha_j: j<i  \}, M)      \}
\]
is stationary in $\mu^+$,

\item[(d)] there exists $\chi$ with $\lambda > \chi=\cf(\chi) > \cf(\lambda)$ such that for every model $M$ with universe $\lambda^+$ and vocabulary of cardinality $\cf(\lambda)$,
we can find an increasing sequence $\vec{\alpha}=\langle \alpha_i: i < \chi \rangle$ of ordinals less than $\lambda^+$ such that
\[
S^M_{\vec{\alpha}}=\{i<\chi: cl(\{\alpha_i\}, M) \cap \lambda \subseteq cl(\{\alpha_j: j<i  \}, M)      \}
\]
is stationary in $\chi$,

\item[(e)] there is no sequence $\vec{X}=\langle U_i: i<\lambda^+ \rangle$ such that each $U_i \cap \lambda$ is a  cofinal subset of  $\lambda$, $U_i \cap \lambda$ has size $\cf(\lambda)$,
and for every $i<\lambda^+$ there is a sequence $\vec{X}_i=\langle (\alpha_{i, j}, \beta_{i, j}): j<i       \rangle$ such that:
\begin{itemize}
\item $\vec{X}_i$ has no repetition,
\item $\alpha_{i, j} \in U_i$,
\item $\beta_{i, j} \in U_j \cap \lambda$.
\end{itemize}
\end{enumerate}
Furthermore,  the statement (e) is equivalent to $\neg \UB_\lambda$, provided that $\cf(\lambda)$ is not a Jonsson cardinal.
\end{theorem}
\begin{remark}
The assumption ``$\lambda$ is a strong limit cardinal'' is only used in the proof of (e) implies $\neg \UB_\lambda$.
\end{remark}
\begin{proof}
We prove the theorem by a sequence of claims. First note that:
\begin{claim}
\label{cla5}
Clause (a) is a special case of clause (b), and
clause (c) implies clause (d).
 \end{claim}
\begin{claim}
\label{cl2}
(b) implies (c).
\end{claim}
\begin{proof}
Let $M$ be a model with universe $\lambda^+$ and vocabulary of cardinality at most $\cf(\lambda)$. By (b), there exists an elementary submodel $N \prec M$ such that $||N||=\mu^+$ and $|N \cap \lambda|=\mu.$ Let $\vec{\alpha}=\langle  \alpha_i: i<\mu^+  \rangle$
list in increasing order the first $\mu^+$ elements of $N.$ So for $i<\mu^+$ we have
\[
cl(\{\alpha_i\}, M) \cap \lambda \subseteq N \cap \lambda,
\]
and since $N \cap \lambda$ has size $\mu$, we can find some $i(*)< \mu^+$ such that
\[
\forall i<\mu^+,~ cl(\{\alpha_i\}, M) \cap \lambda \subseteq \bigcup\limits_{j<i(*)}cl(\{\alpha_j\}, M).
\]
Hence the set $S^M_{\vec{\alpha}}$ includes $[i(*), \mu^+)$ and so is stationary in $\mu^+$,
as requested.
\end{proof}
\begin{claim}
\label{cl3}
(d) implies (e).
\end{claim}
\begin{proof}
Suppose towards a contradiction that (d) holds but (e) fails.
As (e) fails,  we can find  sequences $\vec{X}=\langle U_i: i<\lambda^+ \rangle$
and $\vec{X}_i=\langle (\alpha_{i, j}, \beta_{i, j}): j<i       \rangle$  as in clause (e). Let $M$ be a model in a vocabulary $\mathcal L$ such that:
\begin{enumerate}
\item $|\mathcal L|=\cf(\lambda)$,
\item $M$ has universe $\lambda^+$,
\item $M=(\lambda^+, \langle \tau^M_i: i<\cf(\lambda)\rangle, H^M        )$, where
\begin{enumerate}
\item $\tau^M_i=i,$
\item $H^M$ is a 2-place function such that for all $i$, $U_i \cap \lambda= \{ H^M(i, \alpha): \alpha < \cf(\lambda)    \}$.
\end{enumerate}
\end{enumerate}
Now by (d) applied to the model $M$, we can find a sequence $\vec{\zeta}= \langle \zeta_i: i<\chi \rangle$ of ordinals less than $\lambda^+$ such that the set
$S^M_{\vec{\zeta}}$ is stationary in $\chi$. Let
$\zeta=\sup\limits_{i<\chi}\zeta_i$. Consider the sequence $\vec{X}_\zeta=\langle (\alpha_{\zeta, \xi}, \beta_{\zeta, \xi}): \xi < \zeta       \rangle$.

For $i<\chi$, let
\[
W_i=cl(\{\zeta_j: j< i\}, M) \cap \lambda.
\]
So $\langle W_i: i<\chi \rangle$ is a $\subseteq$-increasing continuous sequence of sets each of cardinality $<\chi$.
Note that for each $i \in S^M_{\vec{\zeta}}$,
\[
\beta_{\zeta, \zeta_i} \in U_{\zeta_i} \cap \lambda \subseteq cl(\{\zeta_i\}, M) \cap \lambda \subseteq W_i.
\]
(The former inclusion $\subseteq$ holds because $\cf(\lambda) \cup \{\zeta_i  \} \subseteq cl(\{\zeta_i\}, M)$ and $cl(\{\zeta_i\}, M)$ is closed under $H^M$. The latter inclusion $\subseteq$ holds because $i \in S^M_{\vec{\zeta}}$). Then since $S^M_{\vec{\zeta}}$ is stationary in $\chi$,
there is $\beta_*$ such that
\[
U=\{i \in S^M_{\vec{\zeta}}: \beta_{\zeta, \zeta_i}=\beta_*  \}
\]
is stationary. Moreover, since $|U_\zeta|=\cf(\lambda) < \chi,$
we get some
$i_i < i_2$ in $U$ such that $\alpha_{\zeta, \zeta_{i_1}}=\alpha_{\zeta, \zeta_{i_2}}$. This contradicts that $\vec{X}_\zeta$
has no repetition.
\end{proof}
\begin{claim}
\label{cl4}
(e) implies $\neg \UB_\lambda$.
\end{claim}
\begin{proof}
Suppose not. Thus we can assume that both (e) and $\UB_\lambda$ hold. Let $f: [\lambda^+]^{<\omega} \to \lambda^+$
witness $\UB_\lambda.$ Choose a vocabulary $\mathcal L$ of size $\cf(\lambda)$ and an  $\mathcal L$-model $M$ such that:
\begin{enumerate}
\item $M$ has universe $\lambda^+,$
\item $M$ expands the model $M_0$ of Lemma \ref{lem2}, by expanding $\mathcal L_0$ (the vocabulary of $M_0$) using the constant symbols $\langle d^M_i: i< \cf(\lambda) \rangle$ and the function symbols
$( \langle F^M_n:  n< \omega \rangle, p^M, G_1^M, G_2^M    )$,
where:
\begin{enumerate}
\item $d_i^M=i$ for $i<\cf(\lambda)$,

\item $F^M_n$ is an $n$-ary function such that
\[
F^M_n(\alpha_0, \cdots, \alpha_{n-1}) = f(\{\alpha_0, \cdots, \alpha_{n-1}    \}),
\]

\item $p^M$ is a pairing function on $\lambda^+,$ mapping $\lambda \times \lambda$ onto $\lambda$

\item $G_1^M$ and $G_2^M$ are 2-place functions such that for every $\alpha \in [\lambda, \lambda^+)$,
$\langle G_1(\beta, \alpha): \beta < \alpha   \rangle$ enumerates $\lambda$ and
\[
\big( \beta < \alpha ~\& ~ \gamma=G_1(\beta, \alpha) \big) \Rightarrow \beta=G_2(\gamma, \alpha).
\]
\end{enumerate}
\end{enumerate}
By expanding $M$ further, let us suppose that
\begin{enumerate}
\item[(3)] $M$ contains Skolem functions.
 \end{enumerate}
For $\alpha < \lambda^+,$ set $N_\alpha= cl(\{\alpha\}, M)$.

$(*)_1$
 $N_\alpha$ belongs to $[\lambda^+]^{\cf(\lambda)}$ and it contains an unbounded subset of $\lambda$.
\begin{proof}
As $\mathcal L$ has size $\cf(\lambda)$, so $|N_\alpha| \leq \cf(\lambda)$. On the other hand, by clause (2)(a), $\cf(\lambda) \subseteq N_\alpha$
and hence $N_\alpha$ belongs to $[\lambda^+]^{\cf(\lambda)}$. Also as $\{c_i^{M_0}: i < \cf(\lambda)   \} \subseteq N_\alpha$ (see the proof of Lemma \ref{lem2}) and $\langle c_i^{M_0}: i < \cf(\lambda)  \rangle$ is an unbounded sequence in $\lambda$, we have  $N_\alpha$ contains an unbounded subset of $\lambda$.
\end{proof}

Let
\[
E=\{ \delta \in (\lambda, \lambda^+): \delta = cl(\delta, M)    \}.
\]
$E$ is clearly a club of $\lambda^+$ and $E \cap \lambda=\emptyset.$
By Lemma \ref{lem2}, we have

$(*)_2$ Suppose $\xi < \zeta$ are in $E$.  Then
\[
\xi \in cl\bigg(\{\zeta\} \cup (N_\xi \cap \lambda) \cup (N_\zeta \cap \lambda), M \bigg).
\]
\begin{proof}
Suppose by the way of contradiction that $\xi \notin cl\big(\{\zeta\} \cup (N_\xi \cap \lambda) \cup (N_\zeta \cap \lambda), M \big).$
Let $V_1=N_\xi \cap \lambda$, $V_2=N_\zeta \cap \lambda$ and $V=V_1 \cup V_2$. By our assumption,
$$\xi \notin cl(\{\zeta\} \cup V, M),$$
also,  it is clear that
$$\zeta \notin cl(\{\xi\} \cup V, M).$$
Thus by Lemma \ref{lem2}, we can find submodels $N^*_1, N^*_2$ of $M$ such that
\begin{enumerate}
\item $N^*_1 \cap \lambda=N^*_2 \cap \lambda$ is unbounded in $\lambda$,
\item $\xi \in N^*_1 \setminus N^*_2$ and $\zeta \in N^*_2 \setminus N^*_1$.
\end{enumerate}
The models $N_1^*$ and $N^*_2$ are clearly $f$-closed, and
by clause (1) above and $\UB_\lambda$, we have $N^*_1 \subseteq N^*_2$ or $N^*_2 \subseteq N^*_1$,
which contradicts clause (2) above.
\end{proof}

Let $\langle   \sigma_i(x_0, \cdots, x_{n(i)-1}): i<\cf(\lambda)   \rangle$
list all terms of $\mathcal{L}$. By $(*)_2$, for each $\xi < \zeta$ from $E$,
we can choose some $i(\xi, \zeta) < \cf(\lambda)$ together with sequences
$\vec{a}_{\xi, \zeta}  \in (N_\zeta \cap \lambda)^{<\omega}$
and
$\vec{b}_{\xi, \zeta}  \in (N_\xi \cap \lambda)^{<\omega}$
such that
\begin{center}
$(\oplus)_1$$\hspace{1.0cm}$ $\xi = \sigma_{i(\xi, \zeta)}(\zeta,  \vec{a}_{\xi, \zeta}, \vec{b}_{\xi, \zeta}).$
\end{center}
For $\xi \in E$ set $U_\xi=N_\xi=cl(\{\xi\}, M).$ It follows that $U_\xi=cl(U_\xi, M)$. For
$ \xi < \zeta$ use the pairing function $p^M$ to find $\alpha_{\zeta, \xi}$ and $\beta_{\zeta, \xi}$ such that
$\alpha_{\zeta, \xi}$ codes $\langle i(\xi, \zeta) \rangle ^{\frown} \vec{a}_{\xi, \zeta}$
and  $\beta_{\zeta, \xi}$ codes $\vec{b}_{\xi, \zeta}$.

Now the sequences $$\vec{X}=\langle U_\xi: \xi \in E  \rangle$$

and  $$\langle \langle (\alpha_{\zeta, \xi}, \beta_{\zeta, \xi}): \xi \in \zeta \cap E     \rangle: \zeta \in E  \rangle$$
witness the failure of (e). We get a contradiction and the claim
follows.
\end{proof}
Thus so far we have shown that
\[
(a) \implies (b) \implies (c) \implies (d) \implies (e) \implies \neg \UB_\lambda.
\]
\begin{claim}
Suppose that $\cf(\lambda)$ is not a Jonsson cardinal. Then $\neg \UB_\lambda$
 implies (e).
 \end{claim}
 \begin{proof}
Suppose towards a contradiction that (e) fails and let $\vec{X}=\langle U_i: i<\lambda^+  \rangle$
and $\langle  \vec{X}_i: i< \lambda^+      \rangle$, where $\vec{X}_i=\langle (\alpha_{i, j}, \beta_{i, j}): j<i   \rangle$
as in clause (e) witness this failure. Let $\langle \lambda_i: i<\cf(\lambda) \rangle$
be an increasing sequence cofinal in $\lambda$ and define the function $c: \lambda \to \cf(\lambda)$ as
\[
c(\alpha) =\min\{i<\cf(\lambda): \alpha < \lambda_i \}.
\]
For $\xi < \lambda^+$ let $\langle \gamma_{\xi, i}: i<\cf(\lambda)  \rangle$
enumerate $U_\xi$ such that each element of $U_\xi$ appears cofinally many often. Let $f: [\lambda^+]^{<\omega} \to \lambda^+$ be such that:
\begin{enumerate}
\item if $\xi < \zeta < \lambda^+$, then
\[
f(\alpha_{\zeta, \xi}, \beta_{\zeta, \xi}, \zeta)=\xi,
\]

\item if $\zeta < \lambda^+$ and $\alpha < \lambda$, then for arbitrary
large $j< \cf(\lambda)$, we have
\[
\sup\limits_{i<j}\lambda_i < \alpha < \lambda_j \implies f(\alpha, \zeta)=\gamma_{\zeta, j}.
\]

\item if $A \in [\cf(\lambda)]^{\cf(\lambda)}$, $c(\alpha_i)=i$ for $i \in A$ and $j<\cf(\lambda),$ then for some $n$ and some sequence $\vec\xi =
\langle  \xi_0, \cdots, \xi_{n-1}   \rangle   \in A^n$, we have
    \[
    j =c (f(\alpha_{\xi_0}, \cdots, \alpha_{\xi_{n-1}})).
    \]
\end{enumerate}
Since $\cf(\lambda)$ is not a Jonsson cardinal, we can define such a function $f$\footnote{this assumption is used to guarantee clause (3) in definition of $f$ holds.}.
Let us show that the pair $(f, c)$ witnesses  $\UB_\lambda$ holds,\footnote{We can define a function $\tilde{f}: [\lambda^+]^{<\omega} \to \lambda^+$ which codes $(f, c)$ so that a set is closed under $\tilde f$ if and only if it is closed under both of $f$ and $c$.} which contradicts our  assumption.
To see this, suppose $x, y \subseteq \lambda^+$ are closed under $f$, $x \cap \lambda=y\cap \lambda$ and $\sup(x\cap \lambda)=\lambda.$
Assume towards a contradiction that $x \nsubseteq y$ and $y \nsubseteq x.$ Let $\xi=\min(x \setminus y)$
and $\zeta=\min(y \setminus x)$, and let us suppose that $\xi < \zeta.$

By clause (3), $\cf(\lambda) \subseteq y,$  and then by clause (2), and since $y \cap \lambda$ is cofinal in $\lambda$, we have $U_\zeta \subseteq y.$ Similarly $U_\xi \subseteq x$. As $x \cap \lambda=y \cap \lambda$ and $U_\xi \subseteq \lambda,$ we conclude that $U_\xi \subseteq y$ as well. Thus by item (1), and since $\alpha_{\zeta, \xi}, \beta_{\zeta, \xi}, \zeta \in y$ we have $\xi \in y$,
which contradicts the choice of $\xi \in x \setminus y$. This completes the proof of the claim.
\end{proof}
The theorem follows.
\end{proof}
\begin{remark}
The above proof shows that the following are equivalent:
\begin{enumerate}
\item clause (e) of Theorem \ref{thm1},

\item for each model $M$ with universe $\lambda^+$ and vocabulary of cardinality $\cf(\lambda),$ there are substructures $N_0, N_1$
of $M$ such that $N_0 \cap \lambda=N_1 \cap \lambda$, $N_0 \nsubseteq N_1$ and $N_1 \nsubseteq N_0.$
\end{enumerate}
\end{remark}
As we noticed earlier, it is consistent relative to the existence of large cardinals that  Chang's transfer principle $(\aleph_{\omega+1}, \aleph_\omega) \twoheadrightarrow (\aleph_{1}, \aleph_0)$ holds with $\aleph_\omega$ being strong limit. Hence by our main theorem, we have the following
corollary.
\begin{corollary}
It is consistent, relative to the existence of large cardinals,  that $\UB_{\aleph_\omega}$ fails.
\end{corollary}


\begin{thebibliography}{}
\bibitem{eskew} Eskew, Monroe; Hayut, Yair; On the consistency of local and global versions of Chang's conjecture. Trans. Amer. Math. Soc. 370 (2018), no. 4, 2879-2905.	 Erratum: Trans. Amer. Math. Soc. 374 (2021), no. 1, 753.
	
\bibitem{hayut} Hayut, Yair; Magidor-Malitz reflection. Arch. Math. Logic 56 (2017), no. 3-4, 253-272.
	
\bibitem{levi} Levinski, Jean-Pierre; Magidor, Menachem; Shelah, Saharon; Chang's conjecture for $\aleph_\omega$. Israel J. Math. 69 (1990), no. 2, 161-172.

\bibitem{shelah1008}  Shelah, Saharon; Non-reflection of the bad set for $\check{I}_\theta[\lambda]$ and pcf. Acta Math. Hungar. 141 (2013), no. 1-2, 11–35.

\bibitem{usuba} Usuba, Toshimichi; New combinatorial principle on singular cardinals and normal ideals. MLQ Math. Log. Q. 64 (2018), no. 4-5, 395-408.


 \end{thebibliography}
\end{document}